\documentclass{amsart}

\usepackage{amsrefs}

\theoremstyle{definition}
\newtheorem{defn}{Definition}

\theoremstyle{plain}
\newtheorem{theorem}[defn]{Theorem}
\newtheorem{proposition}[defn]{Proposition}

\newtheorem{corollary}[defn]{Corollary}

\theoremstyle{remark}
\newtheorem{remark}[defn]{Remark}
\newtheorem*{acknowledgement}{Acknowledgement}

\title{The perturbation of the Seiberg-Witten equations revisited}

\author{Mikio Furuta}
\address{Graduate School of Mathematical Sciences, the University of Tokyo, 3-8-1 Komaba, Meguro, Tokyo 153-8914, Japan}
\email{furuta@ms.u-tokyo.ac.jp}
\thanks{The first author is supported in part by Grant-in-Aid for Scientific Research (B) 24340011.}

\author{Shinichiroh Matsuo}
\address{Department of Mathematics, Osaka University, Toyonaka, Osaka 560-0043, Japan}
\email{matsuo@math.sci.osaka-u.ac.jp}
\thanks{The second author is supported in part by Grant-in-Aid for Young Scientists (B) 25800045.}


\DeclareMathOperator{\sw}{SW}
\newcommand{\Z}{\mathbb{Z}}
\newcommand{\spinc}{\mathrm{Spin}^c}
\newcommand{\s}{\mathfrak{s}}
\newcommand{\Hom}{\mathrm{Hom}}
\newcommand{\R}{\mathbb{R}}
\newcommand{\sdforms}{\Lambda^+}
\newcommand{\spinctext}{$\text{spin}^c$}
\newcommand{\W}{Weitzenb\"ock }
\newcommand{\vol}{\mathrm{Vol}}
\newcommand{\C}{\mathbb{C}}
\newcommand{\gauge}{\mathcal{G}}
\newcommand{\conf}{\mathcal{C}}
\newcommand{\quot}{\mathcal{B}}
\newcommand{\map}{\mathrm{Map}}
\newcommand{\U}{\mathrm{U}}
\newcommand{\moduli}{\mathcal{M}}
\renewcommand{\wp}{S^+}
\newcommand{\wm}{S^-}
\newcommand{\wpm}{S^{\pm}}
\newcommand{\st}{s}
\newcommand{\ct}{c}
\newcommand{\sdh}{\omega}
\newcommand{\sdheq}{\hat{\omega}}

\begin{document}
\maketitle

\begin{abstract}
	We introduce a new class of perturbations of the Seiberg-Witten equations.
	Our perturbations offer flexibility in the way the Seiberg-Witten invariants are constructed and also shed a new light to LeBrun's curvature inequalities.
\end{abstract}


\tableofcontents

\section{Introduction} \label{section: introduction}

The Seiberg-Witten invariants, or the monopole invariants, are invariants of a smooth, closed, oriented $4$-manifold $X$.
When $b^+(X)$ is greater than $1$ and a homology orientation for $X$ is fixed, they can be regarded as a map
\[
	\sw \colon \spinc (X) \to \Z,
\]
where $\spinc(X)$ denotes the set of isomorphism classes of \spinctext-structures on $X$.
They are defined, roughly speaking, by counting solutions to the Seiberg-Witten equations on $X$.
In this paper we will introduce a new class of perturbations of the Seiberg-Witten equations.

Perturbations of these equations have played a prominent role in the interplay between the invariant and the equations.
A standard approach of defining the invariant employs a generic self-dual $2$-form to achieve transversality of the equations.
Witten~\cite{MR1306021} deformed the equations by holomorphic $2$-forms to show that the invariants of a Kahler surface are completely described in terms of the complex geometry of the surface.
Taubes~\cites{MR1306023,MR1798809} introduced two classess of perturbations to prove spectacular results on symplectic $4$-manifolds.
Ozsv{\'a}th and Szab{\'o}~\cite{MR1745017} and Mrowka, Ozsv{\'a}th, and Yu~\cite{MR1611061} used connections on the spinor bundle that do not necessarily induce the Levi-Civita connection on the tangent bundle.
Kronheimer and Mrowka~\cite{MR2388043} introduced the blown-up Seiberg-Witten equations.
Bauer~\cite{Bauer} proposed a ``regularised monopole map''.
For the other direction, LeBrun~\cite{MR2039991} considered conformal transformations of the equations to obtain a simple proof of his celebrated curvature inequalities.
The purpose of this paper is to add a new class to the list.

The new perturbations, introduced in Section~\ref{section: psw}, are a natural consequence of the \W formulae for self-dual $2$-forms and the Dirac operator; the key estimate of the article is the inequality (\ref{ineq: key}).
This key inequality also leads us to define an invariant $\lambda_{\theta}$, which characterise the class of almost-K\"ahler metrics (Proposition~\ref{prop: symp kahler}).
The perturbations involve the scalar curvature, the self-dual Weyl curvature, and the invariant $\lambda_{\theta}$.
Our original motivation for modifying the Seiberg-Witten equations in the light of new perturbations comes from LeBrun's curvature inequalities~\cites{MR1359969, MR1872548}, and we slightly improve them (Theorem~\ref{thm: general lebrun}) along with a new proof of his original inequalities, which will be explained in Section \ref{section: LeBrun}.

Before describing our perturbations in detail, we illustrate them by a simple one.
We first set our conventions for it.
Our notation basically follows that of~\cite{MR2388043}.
Let $(X,g)$ be a smooth, closed, oriented, Riemannian $4$-manifold, and $\s$ a \spinctext-structure on $X$.
For simplicity, we assume that $b^+(X) > 1$ in Introduction.
We denote the bundle of self-dual $2$-forms by $\sdforms$, the scalar curvature of $g$ by $R_g$, and the self-dual Weyl curvature of $g$ by $W^+_g$.
We recall that $W^+_g$ at a point $x \in X$ may be viewed as a trace-free endomorphism $W^+_g(x) \colon \sdforms_x \to \sdforms_x$ of the self-dual $2$-forms at $x$, and a Lipschitz continuous function $w_g \colon X \to (-\infty, 0]$ is defined by its lowest eigenvalue.
The \spinctext-structure $\s$ determines a triple $(\wp, \wm, \rho)$, where $\wpm$ are Hermitian $2$-plane bundles and $\rho \colon T^*X \to \Hom (\wp, \wm)$ is the Clifford multiplication.
The determinant line bundle of $\s$ is denoted by $\det(\s)$.
The canonical real-quadratic map is denoted by $\sigma: \wp \to \sdforms$, and satisfies the pointwise equality $|\sigma(\Phi)|^2 = |\Phi|^4/8$.
The self-dual part of the first Chern class $c_1(\s)$ of the determinant line bundle is denoted by $c_1^+(\s)$.
We adhere to the notational convention that, for any real-valued function $f \colon X \to \R$, we define $f_+ \colon X \to [0, \infty)$ and $f_- \colon X \to [0, \infty)$ by $f_+(x) := \max(f(x), 0)$ and $f_-(x) := \max(-f(x), 0)$ respectively.
Our convention differs from that of~\cite{MR1872548}*{p.287}.
For example, $R_g = (R_g)_+ - (R_g)_-$ and $|R_g| = (R_g)_+ + (R_g)_-$.
Note that $f_{\pm}$ might be only Lipschitz continuous even if $f$ is smooth.
Let us also fix a smooth cut-off function $\beta \colon [0,\infty) \to [0,2]$ that satisfies $\beta(t) \le 1/t$ for $t \in [0,\infty)$, $\beta(t) = 1$ for $t \ll 1$, and $\beta(t) = 1/t$ for $t \ge 1$.
Let $\epsilon > 0$.
Now we can write down a simple version of our perturbed Seiberg-Witten equations for a connection $A$ on $\det(\s)$ and a section $\Phi$ of $\wp$:
\begin{equation} \label{psw intro}
	\left\{
	\begin{aligned}
		D_A \Phi &= 0 \\
		iF_A^+ &= - \frac{1}{\sqrt{8}} \beta(|\sigma(\Phi)|) \left[ \left(\frac{2}{3}R_g + 2w_g \right)_- + \epsilon \right] \sigma(\Phi).
	\end{aligned}
	\right.
\end{equation}
This perturbation can be obtained from (\ref{psw}) in Section~\ref{section: psw} by setting $\sin \theta = 1$ and $\sdheq = 0$ and using $\lambda_{\theta} \ge 0$.
We will show in Section~\ref{section: compactness} and~\ref{section: invariants} that the moduli spaces of solutions to these equations are always compact, and that the invariants defined by them coincide with the Seiberg-Witten invariants. 
We thus deduce that, if a \spinctext-structure $\s$ satisfies $\sw(\s) \ne 0$, then we have a solution to these perturbed equations for every Riemannian metric $g$ and any $\epsilon > 0$.
We now emphasise that this fact yields a quick proof of one of LeBrun's curvature inequalities~\cite{MR1872548}*{Theorem 2.4}:
For any \spinctext-structure $\s$ with $\sw(\s) \ne 0$, we have a solution $(A,\Phi)$ to (\ref{psw intro}) for any $\epsilon > 0$.
Then, the second equation implies that
\begin{align*}
	\int_X |iF_A^+|^2 \,d\mu_g &= \int_X \left| - \frac{1}{\sqrt{8}} \beta(|\sigma(\Phi)|) \left[ \left(\frac{2}{3}R_g + 2w_g \right)_- + \epsilon \right] \sigma(\Phi) \right|^2 \,d\mu_g \\
	&\le \frac{1}{8} \int_X \left[ \left(\frac{2}{3}R_g + 2w_g \right)_- + \epsilon \right]^2 \,d\mu_g,
\end{align*}
where we have used $\beta(|\sigma(\Phi)|) |\sigma(\Phi)| \le 1$.
The last inequality gives
\[
	4\pi^2 (c_1^+(\s))^2 \le \int_X |iF_A^+|^2 \,d\mu_g \le \frac{1}{8} \int_X \left[ \left(\frac{2}{3}R_g + 2w_g \right)_- + \epsilon \right]^2 \,d\mu_g
\]
for any $\epsilon > 0$, and hence we conclude that
\[
	(c_1^+(\s))^2 \le \frac{1}{32\pi^2} \int_X \left(\frac{2}{3}R_g + 2w_g \right)_-^2 \,d\mu_g \le \frac{1}{32\pi^2} \int_X \left(\frac{2}{3}R_g + 2w_g \right)^2 \,d\mu_g.
\]
for any \spinctext-structure $\s$ with $\sw(\s) \ne 0$.
We will also reprove in Section~\ref{section: LeBrun} that $g$ is almost K\"ahler if $c_1^+(\s) \ne 0$ and equality holds.

\section{The perturbed Seiberg-Witten equations} \label{section: perturbation}
In this section we introduce our perturbations of the Seiberg-Witten equations in full generality.
Let $(X,g)$ be a smooth, closed, oriented, Riemannian $4$-manifold, and $\s$ a \spinctext-structure on $X$.
We denote the scalar curvature and the self-dual Weyl curvature of $g$ by $R_g$ and $W^+_g$ respectively.
We define a Lipschitz continuous function $w_g \colon X \to (-\infty, 0]$ to be the lowest eigenvalue of the trace-free endomorphism $W^+_g(x) \colon \sdforms_x \to \sdforms_x$ of the self-dual $2$-forms at $x$.
The \spinctext-structure $\s$ determines a triple $(\wp, \wm, \rho)$, where $\wpm$ are Hermitian $2$-plane bundles and $\rho \colon T^*X \to \Hom (\wp, \wm)$ is the Clifford multiplication.
The determinant line bundle of $\s$ is denoted by $\det(\s)$.
The canonical real-quadratic map is denoted by $\sigma: \wp \to \sdforms$, and satisfies the pointwise equality $|\sigma(\Phi)|^2 = |\Phi|^4/8$.
Fix a reference smooth connection $A_0$ on $\det(\s)$.

\subsection{\W formulae}
We begin by proving some inequalities through \W formulae, which will be the key to everything that follows.
Let $\theta \colon X \to (\R / 2\pi\Z)$ be a smooth function.
For brevity, we abbreviate $\sin \theta$ and $\cos \theta$ as $\st$ and $\ct$ respectively.

The \W formula for self-dual $2$-forms reads
\begin{equation*} \label{sd w}
	(d+d^*)^2 \sigma = \nabla^* \nabla \sigma + \frac{1}{3} R_g \sigma - 2W^+(\sigma, \cdot),
\end{equation*}
and it implies that
\begin{equation} \label{eq: s}
	\begin{aligned}
	&\int_X |(d+d^*) (\st \sigma)|^2 \,d\mu_g	= \int_X \big[ |\nabla (\st \sigma)|^2 + \frac{1}{3} R_g |\st \sigma|^2 - 2 W_g^+(\st \sigma, \st \sigma) \big] \,d\mu_g \\
	= &\int_X \big[ \ct^2 |d\theta \otimes \sigma|^2 + \st^2 |\nabla \sigma|^2 + 2\ct \st(d\theta \otimes \sigma,\nabla \sigma) + \frac{1}{3} \st^2 R_g |\sigma|^2 - 2\st^2 W_g^+(\sigma, \sigma) \big] \,d\mu_g
	\end{aligned}
\end{equation}
for any smooth self-dual $2$-form $\sigma$.
We have
\begin{equation} \label{eq: c}
	\int_X |\nabla (\ct \sigma)|^2 \,d\mu_g = \int_X \big[ \st^2 |d\theta \otimes \sigma|^2 + \ct^2 |\nabla \sigma|^2 - 2\ct \st(d\theta \otimes \sigma,\nabla \sigma) \big] \,d\mu_g
\end{equation}
for any smooth self-dual $2$-form $\sigma$.
These two equalities (\ref{eq: s}) and (\ref{eq: c}) combine to give
\begin{equation} \label{ineq: s and c}
	\begin{aligned}
	&\int_X |(d+d^*) (\st \sigma)|^2 \,d\mu_g + \int_X |\nabla (\ct \sigma)|^2 \,d\mu_g \\
	= &\int_X \big[ |d\theta|^2 |\sigma|^2 + \frac{1}{3} \st^2 R_g |\sigma|^2 - 2\st^2 W_g^+(\sigma, \sigma) \big] \,d\mu_g + \int_X |\nabla \sigma|^2 \,d\mu_g \\
	\le &\int_X \big[ |d\theta|^2 |\sigma|^2 + \frac{1}{3} \st^2 R_g |\sigma|^2 - 2\st^2 w_g |\sigma|^2 \big] \,d\mu_g + \int_X |\nabla \sigma|^2 \,d\mu_g
	\end{aligned}
\end{equation}
for any smooth self-dual $2$-form $\sigma$.

On the other hand, the \W formula for the Dirac operator reads
\begin{equation*} \label{dirac w}
	(D_A^* D_A \Phi, \Phi) = \frac{1}{2}\Delta_g |\Phi|^2 + |\nabla_A \Phi|^2 + \frac{1}{4} R_g |\Phi|^2 + 2(-i F_A^+, \sigma(\Phi)),
\end{equation*}
and it implies
\begin{equation} \label{eq: dirac}
	\begin{aligned}
	 &\int_X (D_A \Phi, D_A(|\Phi|^2 \Phi)) \,d\mu_g \\
	= &\int_X \big[ |\Phi|^2 |\nabla_A \Phi|^2 + \frac{1}{2}|\nabla |\Phi|^2|^2 + \frac{1}{4}R_g |\Phi|^4 - 2(iF_A^+, |\Phi|^2\sigma(\Phi)) \big] \,d\mu_g
	\end{aligned}
\end{equation}
for any smooth connection $A$ on $\det (\s)$ and any smooth section $\Phi$ of $\wp$.
Note that the particular self-dual $2$-form $\sigma = \sigma(\Phi)$ for a section $\Phi$ of $\wp$ satisfies the pointwise ``log Kato inequality''
\[
	\frac{|\nabla \sigma(\Phi)|}{|\sigma(\Phi)|} \le 2 \frac{|\nabla_A \Phi|}{|\Phi|},
\]
and our convention is $|\sigma(\Phi)|^2 = |\Phi|^4 / 8$; hence (\ref{eq: dirac}) can be rewritten in the form
\begin{equation}
	\begin{aligned} \label{ineq: dirac w}
	&\int_X |\nabla \sigma|^2 \,d\mu_g \\
	\le &\int_X \big[ -2|\nabla |\sigma||^2 - R_g |\sigma|^2 + (\sqrt{8}iF_A^+, |\sigma|\sigma) + \frac{1}{2} (D_A \Phi, D_A(|\Phi|^2 \Phi)) \big] \,d\mu_g
	\end{aligned}
\end{equation}
for any smooth connection $A$ on $\det (\s)$, any smooth section $\Phi$ of $\wp$, and $\sigma := \sigma(\Phi)$.

Now we piece (\ref{ineq: s and c}) and (\ref{ineq: dirac w}) together to obtain
\begin{equation} \label{ineq: key}
	\begin{aligned}
	\int_X |(d+d^*) (\st \sigma)|^2 \,d\mu_g + \int_X |\nabla (\ct \sigma)|^2 \,d\mu_g + 2\int_X |\nabla |\sigma||^2 \,d\mu_g \\
	\le \int_X - \big[ (1 - \st^2/3)R_g + 2\st^2 w_g - |d\theta|^2 \big] |\sigma|^2 \,d\mu_g + \int_X (\sqrt{8}iF_A^+, |\sigma|\sigma) \,d\mu_g \\ + \frac{1}{2}\int_X (D_A \Phi, D_A(|\Phi|^2 \Phi)) \,d\mu_g
	\end{aligned}
\end{equation}
for any smooth connection $A$ of $\wp$, any smooth section $\Phi$ of $\wp$, and $\sigma := \sigma(\Phi)$.
This inequality will play a pivotal role in the sequel.
The inequality (\ref{ineq: key}) leads us to define a non-negative constant $\lambda_{\theta}$ by
\begin{equation} \label{lambda}
	\lambda_{\theta} := \inf \left\{ \frac{\|(d+d^*) (\st \sigma)\|_2^2 + \|\nabla (\ct \sigma)\|_2^2 + 2\|\nabla |\sigma|\|_2^2}{\|\sigma\|_2^2} \mid 0 \ne \sigma \in L^2_1(\sdforms) \right\},
\end{equation}
and a Lipschitz function $K_{\theta}$ or $K$ on $X$ by
\[
	K_{\theta} := \left( 1 - \frac{1}{3} \sin^2 \theta \right) R_g + 2(\sin^2 \theta) w_g - |d\theta|^2 + \lambda_{\theta}.
\]
The inequality (\ref{ineq: key}) can be reformulated in terms of $K_{\theta}$ as
\begin{equation} \label{ineq: key reformulated}
	\int_X K_{\theta} |\sigma|^2 \,d\mu_g \le \int_X (\sqrt{8}iF_A^+, |\sigma|\sigma) \,d\mu_g + \frac{1}{2}\int_X (D_A \Phi, D_A(|\Phi|^2 \Phi)) \,d\mu_g
\end{equation}
for any smooth connection $A$ on $\det(\s)$, any smooth section $\Phi$ of $\wp$, and $\sigma := \sigma(\Phi)$.
We observe that $\sigma = \sigma(\Phi)$ is in $L^2_1$ for any $L^p_{1,A_0}$-section $\Phi$ of $\wp$ with $p > 4$; therefore, the inequality (\ref{ineq: key reformulated}) holds for any $L^p_{1,A_0}$-connection on $\det (\s)$ and any $L^p_{1,A_0}$-section of $\wp$.

The invariant $\lambda_{\theta}$ characterises the class of almost-K\"ahler metrics.
\begin{proposition} \label{prop: symp kahler}
	Let $(X,g)$ be a smooth, closed, oriented, Riemannian $4$-manifold, and $\theta \colon X \to (\R / 2\pi \Z)$ a smooth function.
	Then, $\lambda_{\theta} = 0$ if and only if $g$ is almost K\"ahler and $\theta$ is a constant function.
	Moreover, if $\cos \theta > 0$, then $g$ is K\"ahler.
\end{proposition}
\begin{proof}
	Assume $\lambda_{\theta} = 0$.
	Then, there exists a sequence $\{\sigma_j\}$ of $g$-self-dual $2$-forms such that $\sigma_j \in L^2_1$, $\|\sigma_j\|_2 = 1$, and
	\[
		\| (d+d^*) (\st \sigma_j)\|_2^2 + \|\nabla (\ct \sigma_j)\|_2^2 + 2\|\nabla |\sigma_j|\|_2^2 \le \frac{1}{j} \big( \|\st \sigma_j \|_2^2 + \|\ct \sigma_j \|_2^2 \big),
	\]
	where $\st = \sin \theta$ and $\ct = \cos \theta$.
	In particular, $\sigma_j$ are uniformly $L^2_1$-bounded by (\ref{eq: s}) and (\ref{eq: c}).
	Thus, there exists a $2$-form $\sigma_{\infty}$ such that $\sigma_j$ strongly $L^2$-converges to $\sigma_{\infty}$, and $(d+d^*)(\st \sigma_j)$ and $\nabla(\ct \sigma_j)$ strongly $L^2$-converge to $(d+d^*)(\st \sigma_{\infty})$ and $\nabla(\ct \sigma_{\infty})$ respectively.
	It follows, therefore, that $\| (d+d^*)(\st \sigma_{\infty})\|_2^2 + \|\nabla(\ct \sigma_{\infty})\|_2^2 + 2\|\nabla |\sigma_{\infty}|\|_2^2 = 0$, and elliptic estimates show that $\sigma_{\infty}$ is a non-trivial smooth $g$-self-dual $2$-form.
	Consequently, $\theta$ is a constant function and $\sigma_{\infty}$ is a symplectic form compatible with $g$.
	Moreover, if $\cos \theta > 0$, then $\sigma_{\infty}$ is $g$-parallel and $g$ is K\"ahler.
	The converse is clear.
\end{proof}

\subsection{Perturbations} \label{section: psw}
We next explain in full generality our perturbations of the Seiberg-Witten equations.
Let $\theta \colon X \to (\R / 2\pi \Z)$ be a smooth function, and we abbreviate $\sin \theta$ and $\cos \theta$ as $\st$ and $\ct$ respectively.
The non-negative constant $\lambda_{\theta}$ is defined by (\ref{lambda}).
Recall that $K_{\theta}$ or $K$ stand for $(1 - \st^2/3) R_g + 2\st^2 w_g - |d\theta|^2 + \lambda_{\theta}$, and that $K_{\pm}$ denotes $\max (\pm K, 0)$.
Fix a smooth cut-off function $\beta \colon [0,\infty) \to [0,2]$ that satisfies $\beta(t) \le 1/t$ for $t \in [0,\infty)$, $\beta(t) = 1$ for $t \ll 1$, and $\beta(t) = 1/t$ for $t \ge 1$.
Let $\sdheq$ be a (not necessarily continuous) $g$-self-dual $2$-form with $\|\sdheq\|_{\infty} \le 1$, and $\epsilon > 0$.
We now consider the following perturbed Seiberg-Witten equations for a connection $A$ on $\det(\s)$ and a section $\Phi$ of $\wp$
\begin{equation} \label{psw}
	\left\{
		\begin{aligned}
		D_A \Phi &= 0 \\
		\sqrt{8}iF_A^+ &= - \beta(|\sigma(\Phi)|) (K_- + \epsilon) \sigma(\Phi) + K_+ \sdheq.
		\end{aligned}
	\right.
\end{equation}
A solution $(A,\Phi)$ of (\ref{psw}) is called reducible if $\Phi = 0$.
The gauge group $\map(X,\U(1))$ acts on the set of solutions.
We remark that, by the choice of $\beta$,
\[
	\beta(|\sigma(\Phi)|) \cdot |\sigma(\Phi)| \le 1
\]
at each point.

We then set up suitable function spaces to define moduli spaces of solutions to (\ref{psw}).
Since we have allowed $\sdheq$ to be of just $L^{\infty}$, it is not expected that solutions are smooth.
We pick a $p > 4$ and a smooth connection $A_0$ on $\det(\s)$, and define
\[
	\conf := \{(A,\Phi) \mid A-A_0 \in L^p_1 \text{ and } \Phi \in L^p_{1,A_0}\}
\]
and
\[
	\gauge := \{u \colon X \to \C \mid u \in L^p_2 \text{ and } |u| = 1\}.
\]
The gauge group $\gauge$ is a Banach Lie group acting smoothly on the configuration space $\conf$.
Let $\quot := \conf / \gauge$, and write $[A,\Phi]$ for the orbit of $(A,\Phi)$ under $\gauge$.
We then define the moduli space $\moduli(\theta,\sdheq,\epsilon)$ to be
\[
	\moduli(\theta,\sdheq,\epsilon) := \{ [A,\Phi] \in \quot \mid \text{ (\ref{psw}) holds.} \}.
\]
Our moduli spaces are always compact, which will be proved in the next section.

\subsection{Compactness} \label{section: compactness}
In this section we prove that moduli spaces for the perturbed Seiberg-Witten equations (\ref{psw}) are compact.
The main idea that underlies our proof is to convert quantitative estimates in LeBrun~\cite{MR1872548} to a qualitative property of compactness of moduli spaces.
Fix a smooth function $\theta \colon X \to (\R / 2\pi \Z)$, a $g$-self-dual $2$-form $\sdheq$ with $\|\sdheq\|_{\infty} \le 1$, and a constant $\epsilon > 0$, and we consider the perturbed Seiberg-Witten equations (\ref{psw}) for these $\theta$, $\sdheq$, and $\epsilon$ throughout this section.

We first observe that the $L^{\infty}$-bound on $F_A^+$ is immediate.
\begin{proposition} \label{prop: curvature bound}
	Any $L^p_1$-solution $(A, \Phi)$ of (\ref{psw}) satisfies the $L^{\infty}$-bound
	\[
		\|F_A^+\|_{\infty} \le \frac{1}{\sqrt{8}} (\|K\|_{\infty} + \epsilon).
	\]
\end{proposition}
\begin{proof}
	The second equation of~{(\ref{psw})} implies the following pointwise inequalities
	\begin{equation*}
		\begin{aligned}
		\left| \sqrt{8}iF_A^+ \right| &= \left| - \beta(|\sigma(\Phi)|) (K_- + \epsilon)\sigma(\Phi) + K_+ \sdheq \right| \\
		&\le  (K_-+\epsilon) \cdot |\beta(|\sigma(\Phi)|) \sigma(\Phi)| + K_+ \cdot |\sdheq| \\
		&\le K_- + \epsilon + K_+  = |K| + \epsilon,
		\end{aligned}
	\end{equation*}
	where we have used $\beta(|\sigma(\Phi)|) |\sigma(\Phi)| \le 1$, $\|\sdheq\|_{\infty} \le 1$, and $|K| = K_+ + K_-$.
\end{proof}

We next derive the $L^4$-bound on $\Phi$ via the inequality (\ref{ineq: key reformulated}).
\begin{proposition} \label{prop: Phi L4}
	Any $L^p_1$-solution $(A, \Phi)$ of (\ref{psw}) satisfies the $L^4$-bound
	\[
		\int_X |\Phi|^4 \,d\mu_g \le 8 \left( 1 + \frac{\max K_-}{\epsilon} \right) \vol(X,g).
	\]
\end{proposition}
\begin{proof}
	We abbreviate $\sigma(\Phi)$ as $\sigma$.
	The equations (\ref{psw}) and the inequality (\ref{ineq: key reformulated}) combine to yield
	\begin{equation} \label{ineq: key lambda}
		\begin{aligned}
		\int_X (K_+ - K_-) |\sigma|^2 \,d\mu_g 
		&\le \int_X \big( -\beta(|\sigma|) (K_-+\epsilon) \sigma + K_+ \sdheq, |\sigma|\sigma \big) \,d\mu_g \\
		&\le \int_X - \beta(|\sigma|)|\sigma| \cdot (K_- + \epsilon) |\sigma|^2 \,d\mu_g + \int_X K_+ |\sigma|^2 \,d\mu_g,
		\end{aligned}
	\end{equation}
	where we have used $\| \sdheq \|_{\infty} \le 1 $ and $K = K_+ - K_-$.
	Rearranging these inequalities, we have
	\[
		\int_X \beta(|\sigma|) |\sigma| \cdot (K_- + \epsilon) |\sigma|^2 \,d\mu_g \le \int_X K_- |\sigma|^2 \,d\mu_g.
	\]
	Since $\beta(|\sigma|) |\sigma| = 1$ at every point where $|\sigma| \ge 1$, we have
	\[
		\epsilon \int_{|\sigma| \ge 1} |\sigma|^2 \,d\mu_g \le \int_{|\sigma| \le 1} K_- |\sigma|^2 \,d\mu_g.
	\]
	Consequently, we obtain
	\begin{equation*}
		\begin{aligned}
		\int_X |\sigma|^2 \,d\mu_g =& \int_{|\sigma| \le 1} |\sigma|^2 \,d\mu_g + \int_{|\sigma| \ge 1} |\sigma|^2 \,d\mu_g \\
		\le& \int_{|\sigma| \le 1} |\sigma|^2 \,d\mu_g +\frac{1}{\epsilon} \int_{|\sigma| \le 1} K_- |\sigma|^2 \,d\mu_g.
		\end{aligned}
	\end{equation*}
	The desired estimate follows.
\end{proof}

It is now straightforward to show that our moduli spaces are compact.
\begin{theorem} \label{thm: compactness}
	Let $(X,g)$ be a smooth, closed, oriented, Riemannian $4$-manifold, and $\s$ a \spinctext-structure on $X$.
	Fix a $p > 4$.
	For any smooth function $\theta \colon X \to (\R / 2\pi \Z)$, any $g$-self-dual $2$-form $\sdheq$ with $\|\sdheq\|_{\infty} \le 1$, and any $\epsilon > 0$, the moduli space $\moduli(\theta, \sdheq, \epsilon)$ for the perturbed Seiberg-Witten equations~{(\ref{psw})} is compact.
\end{theorem}
\begin{proof}
	Proposition~\ref{prop: curvature bound} gives the $L^p_1$-bound on $A$ in an appropriate gauge for any $p$.
	Then, Proposition~\ref{prop: Phi L4} and the first equation gives the $L^4_{1, A_0}$-bound on $\Phi$, and it again gives $L^4_{2, A_0}$-bound on $\Phi$.
	In particular, we obtain the $L^q$-bound on $\Phi$ for any $q$.
	The first equation finally provides the $L^p_{1,A_0}$-bound on $\Phi$.
	Compactness of $\moduli(\theta, \sdheq, \epsilon)$ follows.
\end{proof}

Actually, we have proved more.
\begin{theorem} \label{thm: family compactness}
	Let $(X,g)$ be a smooth, closed, oriented, Riemannian $4$-manifold, and $\s$ a \spinctext-structure on $X$.
	Fix a $p > 4$ and a smooth function $\theta \colon X \to (\R / 2\pi\Z)$.
	Let $F \colon X \times [0,\infty) \to \R$ be a $L^{\infty}$-function, and $\eta$ a (not necessarily continuous) $g$-self-dual $2$-form.
	Assume that $|\eta(x)| \le (K_{\theta})_+ (x)$ for any $x \in X$, and that there exist a constant $T > 0$, $\kappa > 0$, and $\delta > 0$ such that
	\begin{equation*}
		\kappa \ge t F(x,t) \ge
		\begin{cases}
			0  &\text{if $(x,t) \in X \times [0,T]$} \\
			(K_{\theta})_- (x) + \delta&\text{if $(x,t) \in X \times [T, \infty)$.}
		\end{cases}
	\end{equation*}
	Consider the following perturbed Seiberg-Witten equations
	\begin{equation} \label{general psw}
		\left\{
			\begin{aligned}
			D_A \Phi &= 0 \\
			\sqrt{8}iF_A^+ &= - F(x,|\sigma(\Phi)|) \sigma(\Phi) + \eta.
			\end{aligned}
		\right.
	\end{equation}
	Then, the moduli spaces of $L^p_1$-solutions to (\ref{general psw}) are compact.
	Every $L^p_1$-solution to (\ref{general psw}) is smooth if both $F$ and $\eta$ are smooth.
\end{theorem}
The equations (\ref{psw}) can be recovered from (\ref{general psw}) by setting $F(x,t) := \beta(t) (K_-(x) + \epsilon)$ and $\eta := K_+ \sdheq$.

\begin{remark}
	Sung~\cite{MR1923274}*{Theorem 3.4} has discovered that there exist an almost-K\"ahler metric $g_h$ on $T^2 \times \Sigma$, where $T^2$ is a torus and $\Sigma$ is a closed Riemannian surface of genus greater than $1$, and a constant $\delta' > 1$ such that
	\[
		\int_X ((1-\delta'/3)R_{g_h} + 2\delta' w_{g_h})^2 \,d\mu_{g_h} < 32\pi^2 (c_1^+(T^2 \times \Sigma))^2.
	\]
	This example illustrates that, for any $\epsilon > 0$, there exists an $\epsilon' \in (0,\epsilon]$ such that the moduli space of solutions to the following equations on $(T^2 \times \Sigma, g_h)$
	\begin{equation*}
		\left\{
			\begin{aligned}
			D_A \Phi &= 0 \\
			\sqrt{8}iF_A^+ &= - \beta(|\sigma(\Phi)|) \big[ (1-\delta'/3)R_{g_h} + 2\delta' w_{g_h})_- + \epsilon' \big] \sigma(\Phi)
			\end{aligned}
		\right.
	\end{equation*}
	is not compact.
	See also \cite{MR1822361}*{Theorem 4.3} and \cite{MR2681684}.
\end{remark}

\subsection{Invariants} \label{section: invariants}
In this section we show that the invariant defined by the perturbed Seiberg-Witten equations~{(\ref{psw})} coincides with the Seiberg-Witten invariant.

Ruan's virtual neighbourhood technique~\cite{MR1635698} (cf.~\cite{MR2025298}*{Proposition 3.3}) works for the perturbed Seiberg-Witten equations (\ref{psw}) because we have shown in Theorem~\ref{thm: compactness} that their moduli spaces $\moduli(\theta,\sdheq,\epsilon)$ are always compact, and we can thus extract integer-valued invariants from $\moduli(\theta,\sdheq,\epsilon)$.
We remark that, if $b_1(X) > 0$, we consider a Banach bundle over the Picard torus and need a $C^1$-partition of unity in $L^p_1$, which always exists (see~\cite{MR1904712}, for example).

We next show that the invariant defined by the perturbed Seiberg-Witten equations~{(\ref{psw})} coincides with the Seiberg-Witten invariant.
If $b^+(X) = 1$, then the Seiberg-Witten invariant depends on a chamber structure in the space of $g$-self-dual $2$-forms $\Omega^+_g(X)$; in this case, we assume that the perturbed Seiberg-Witten equations (\ref{psw}) does not admit any reducible solutions and we only consider Seiberg-Witten invariants for a chamber that contains $K_+ \sdheq$.
\begin{theorem}
	Let $(X,g)$ be a smooth, closed, oriented, Riemannian $4$-manifold, and $\s$ a \spinctext-structure on $X$.
	Fix a $p > 4$, a smooth function $\theta \colon X \to (\R / 2\pi \Z)$, a $g$-self-dual $2$-form $\sdheq$ with $\|\sdheq\|_{\infty} \le 1$, and an $\epsilon > 0$.
	If $b^+(X) = 1$, we assume that the perturbed Seiberg-Witten equations (\ref{psw}) does not admit any reducible solutions.
	Then, the Seiberg-Witten invariant coincides with the invariant defined by (\ref{psw}).
\end{theorem}
\begin{proof}
Let $C$ be a positive constant larger than $100(\|\eta\|_{\infty} + \|R_g\|_{\infty})$ and $(\|K_-\|_{\infty} + 1/100)$, and $\eta$ a smooth $2$-form that satisfies $|\eta(x)| \le K_+(x)$ for any $x \in X$ and belongs to the same chamber as $K_+ \sdheq$.

We first note that, via rescaling $\hat{g} = (C/\sqrt{8})^{-2} g$, an $L^p_1$-solution to the Seiberg-Witten equations
\begin{equation} \label{sw}
	\left\{
		\begin{aligned}
		D_A \Phi &= 0 \\
		iF_A^+ &= - \sigma(\Phi) + \frac{1}{\sqrt{8}}\eta
		\end{aligned}
	\right.
\end{equation}
with respect to $\hat{g}$ and $\s$ is put into one-to-one correspondence with an $L^p_1$-solution to the perturbed Seiberg-Witten equations
\begin{equation} \label{rescaled SW}
	\left\{
	\begin{aligned}
		D_A \Phi &= 0 \\
		\sqrt{8} iF_A^+ &= - C \sigma(\Phi) + \eta
	\end{aligned}
	\right.
\end{equation}
with respect to $g$ and $\s$.
If (\ref{psw}) does not admit any reducible solutions, then neither (\ref{sw}) nor (\ref{rescaled SW}) admits any reducible ones.
In particular, the Seiberg-Witten invariant, which is defined by (\ref{sw}), coincides with the invariant defined by (\ref{rescaled SW}).

Since $\eta$ is smooth, any $L^p_1$-solution $(A,\Phi)$ to (\ref{rescaled SW}) is smooth; hence, the maximum principle yields $\|\sigma(\Phi)\|_{\infty} \le 1$.
Consequently, each $L^p_1$-solution $(A,\Phi)$ to (\ref{rescaled SW}) satisfies the following perturbed Seiberg-Witten equations
\begin{equation} \label{rescaled cutoff SW}
	\left\{
	\begin{aligned}
		D_A \Phi &= 0 \\
		\sqrt{8} iF_A^+ &= - C \gamma(|\sigma(\Phi)|) \sigma(\Phi) + \eta,
	\end{aligned}
	\right.
\end{equation}
where $\gamma \colon [0,\infty) \to [0,1]$ is a smooth cut-off function with $\gamma(t) = 1$ for $t \le 1$ and $1/2t \le \gamma(t) \le 2/t$ for $1 \le t$.
Any $L^p_1$-solution $(A,\Phi)$ to (\ref{rescaled cutoff SW}) is also smooth; hence, the maximum principle again $\|\sigma(\Phi)\|_{\infty} \le 1$.
Consequently, each $L^p_1$-solution to (\ref{rescaled cutoff SW}) satisfies (\ref{rescaled SW}).
In particular, the invariant defined by (\ref{rescaled SW}) coincides with the one defined by (\ref{rescaled cutoff SW}).

A pair $(F,\eta)$ with $F(x,t) := C\gamma(t)$ satisfies the assumption of Theorem \ref{thm: family compactness}.
Take a path of pairs $(F_t, \eta_t)$ from (\ref{psw}) to (\ref{rescaled cutoff SW}), each of which satisfies the assumptions of Theorem~\ref{thm: family compactness}.
Ruan's virtual neighbourhood technique again works for (\ref{general psw}) with $(F_t, \eta_t)$, and the invariant defined by (\ref{psw}) coincides with the one defined by (\ref{rescaled cutoff SW}).
The theorem follows.
\end{proof}

\section{LeBrun's curvature inequalities} \label{section: LeBrun}
In this section we give yet another proof of LeBrun's curvature inequalities~\cites{MR1822361, MR1919897, MR1359969, MR1872548, MR1487727}.
The non-negative constant $\lambda_{\theta}$ is defined by (\ref{lambda}).
Recall that $K_{\theta}$ or $K$ stand for $(1 - \st^2/3) R_g + 2\st^2 w_g - |d\theta|^2 + \lambda_{\theta}$ and that $K_{\pm}$ denotes $\max (\pm K, 0)$.

\begin{theorem} \label{thm: general square lebrun}
	Let $(X,g)$ be a smooth, closed, oriented, Riemannian $4$-manifold with $b^+(X) > 0$ and $\s$ a \spinctext-structure on $X$; in case $b^+(X) = 1$, we assume that $c_1^+(\s) \ne 0$.
	If $\sw(\s) \ne 0$, then we have
	\[
		 32\pi^2 (c_1^+(\s))^2 \le \int_X \left[ \left( (1 - \sin^2 \theta/3) R_g + 2(\sin^2 \theta) w_g - |d\theta|^2 + \lambda_{\theta} \right)_- \right]^2 \,d\mu_g.
	\]
	for any smooth function $\theta \colon X \to (\R/2\pi\Z)$.
\end{theorem}
\begin{proof}
	By assumption, we have a solution $(A,\Phi)$ to the perturbed Seiberg-Witten equations
	\begin{equation*}
		\left\{
			\begin{aligned}
			D_A \Phi &= 0 \\
			\sqrt{8}iF_A^+ &= - \beta(|\sigma(\Phi)|) (K_- + \epsilon) \sigma(\Phi)
			\end{aligned}
		\right.
	\end{equation*}
	for any $\epsilon > 0$.
	The second equation implies that
	\begin{equation*}
		\begin{aligned}
			\int_X |\sqrt{8} iF_A^+|^2 \,d\mu_g =& \int_X \left| - \beta(|\sigma(\Phi)|) (K_- + \epsilon) \sigma(\Phi) \right|^2 \,d\mu_g \\
			\le& \int_X (K_- + \epsilon)^2 \,d\mu_g.
		\end{aligned}
	\end{equation*}
	We have, thus, 
	\[
		32\pi^2 (c_1^+(\s))^2 \le \int_X |\sqrt{8} iF_A^+|^2 \,d\mu_g \le \int_X (K_- + \epsilon)^2 \,d\mu_g
	\]
	for any $\epsilon > 0$.
	The desired inequality follows.
\end{proof}

\begin{theorem} \label{thm: general lebrun}
	Let $(X,g)$ be a smooth, closed, oriented, $4$-manifold, and $\sdh$ a $g$-self-dual harmonic $2$-form.
	Let $\s$ be a \spinctext-structure on $X$ with $\sw(\s) \ne 0$; in case $b^+(X) = 1$, we consider a chamber that contains $K_+ \sdh / |\sdh|$.
	Then, we have
	\[
		 \int_X \left( (1 - \sin^2 \theta/3) R_g + 2(\sin^2 \theta) w_g - |d\theta|^2 + \lambda_{\theta} \right) \frac{|\sdh|_g}{\sqrt{2}} \,d\mu_g \le 4\pi c_1(\s) \cdot [\sdh]
	\]
	for any smooth function $\theta \colon X \to (\R/2\pi\Z)$.
\end{theorem}
\begin{proof}
	We define $\sdheq(x) := \sdh (x) / |\sdh (x)|$ for $x \in X$ and adopt the convention that $\sdheq = 0$ at a point where $\sdh = 0$.
	Then, $\sdheq$ is a $g$-self-dual $2$-form with $\|\sdheq\|_{\infty} \le 1$.

	By assumption, we have a solution $(A,\Phi)$ to the perturbed Seiberg-Witten equations
	\begin{equation} \label{psw lebrun}
		\left\{
			\begin{aligned}
			D_A \Phi &= 0 \\
			\sqrt{8}iF_A^+ &= - \beta(|\sigma(\Phi)|) (K_- + \epsilon) \sigma(\Phi) + K_+ \sdheq
			\end{aligned}
		\right.
	\end{equation}
	for any $\epsilon > 0$.
	The second equation implies that
	\begin{align*}
		\int_X \sqrt{8} iF_A^+ \wedge \sdh &= \int_X \left[ -  \beta(|\sigma(\Phi)|) (K_- + \epsilon) \sigma(\Phi) + K_+ \sdheq \right] \wedge \sdh\\
		&= -  \int_X \beta(|\sigma(\Phi)|) (K_- + \epsilon) \sigma(\Phi) \wedge \sdh + \int_X K_+ \sdheq \wedge \sdh.
	\end{align*}
	Since $(K_- + \epsilon) \ge 0$ and $\beta(|\sigma(\Phi)|) \ge 0$, the Cauchy-Schwartz inequality yields that
	\begin{equation*}
		\begin{aligned}
		- \int_X \beta(|\sigma(\Phi)|) (K_- + \epsilon) \sigma(\Phi) \wedge \sdh 
		&\ge - \int_X \beta(|\sigma(\Phi)|) (K_- + \epsilon) |\sigma(\Phi)| \cdot |\sdh| \,d\mu_g \\
		&\ge - \int_X (K_- + \epsilon) |\sdh| \,d\mu_g.
		\end{aligned}
	\end{equation*}
	Since $\sdh$ is $g$-self-dual harmonic, we have
	\[
		\int_X K_+ \sdheq \wedge \sdh = \int_X K_+ \frac{(\sdh, \sdh)}{|\sdh|} \,d\mu_g = \int_X K_+ |\sdh| \,d\mu_g
	\]
	and
	\[
		\int_X \sqrt{8}iF_A^+ \wedge \sdh = \int_X \sqrt{8}iF_A \wedge \sdh = 4\sqrt{2}\pi c_1(\s) \cdot [\sdh].
	\]
	Consequently, we have
	\[
		 4\sqrt{2}\pi c_1(\s) \cdot [\sdh] \ge - \int_X (K_- + \epsilon) |\sdh| \,d\mu_g + \int_X K_+ |\sdh| \,d\mu_g = \int_X K |\sdh| \,d\mu_g - \epsilon \int_X |\sdh| \,d\mu_g
	\]
	for any $\epsilon > 0$.
	The desired inequality follows.
\end{proof}

Since we always have $\lambda_{\theta} \ge 0$, we can recover the following curvature inequality of LeBrun by taking $\theta$ as a constant function with $\sin^2 \theta = \delta$.
\begin{corollary}[LeBrun] \label{theorem: LeBrun}
	Let $(X,g)$ be a smooth, closed, oriented, $4$-manifold, and $\sdh$ a non-trivial $g$-self-dual harmonic $2$-form.
	Let $\s$ be a \spinctext-structure on $X$ with $\sw(\s) \ne 0$; in case $b^+(X) = 1$, we consider a chamber that contains $K_+ \sdh / |\sdh|$.
	Then, we have
	\[
		\int_X ((1-\delta/3)R_g + 2\delta w_g) \frac{|\sdh|_g}{\sqrt{2}} \,d\mu_g \le 4\pi c_1(\s) \cdot [\sdh]
	\]
	for any $\delta \in [0,1]$.
\end{corollary}

We now proceed to examine when equality holds in Corollary \ref{theorem: LeBrun}.
We first give a quick proof that $g$ is almost-K\"ahler if equality holds in Corollary~\ref{theorem: LeBrun} and $[\sdh] \ne 0$.
If equality holds in Corollary~\ref{theorem: LeBrun}, then Theorem~\ref{thm: general lebrun} implies that
	\[
		\int_X ((1 - \delta/3) R_g + 2\delta w_g + \lambda_{\theta}) \frac{|\sdh|_g}{\sqrt{2}} \,d\mu_g \le \int_X ((1-\delta/3)R_g + 2\delta w_g) \frac{|\sdh|_g}{\sqrt{2}} \,d\mu_g,
	\]
where we have used $d\theta = 0$.
Since $\lambda_{\theta} \ge 0$, we have
	\[
		\lambda_{\theta} \int_X |\sdh| \,d\mu_g = 0.
	\]
If $[\sdh] \ne 0$, we obtain $\lambda_{\theta} = 0$.
Consequently, by Proposition~\ref{prop: symp kahler}, it follows that $g$ is almost K\"ahler and that $g$ is K\"ahler if $\cos ^2 \theta = 1 - \delta > 0$.

\begin{theorem}[LeBrun] \label{theorem: LeBrun equality}
	Let $(X,g)$ be a smooth, closed, oriented, $4$-manifold, and $\sdh$ a non-trivial $g$-self-dual harmonic $2$-form.
	Let $\s$ be a \spinctext-structure on $X$ with $\sw(\s) \ne 0$; in case $b^+(X) = 1$, we consider a chamber that contains $K_+ \sdh / |\sdh|$.
	Equality holds in Corollary~\ref{theorem: LeBrun} if and only if $g$ is almost K\"ahler and $\sdh$ is a positive constant multiple of the compatible symplectic form of $g$.
	Moreover, if $\delta < 1$, then $g$ is K\"ahler.
\end{theorem}
\begin{proof}
	Assume that equality holds in Corollary~\ref{theorem: LeBrun}.
	Let $\theta$ be a constant function on $X$ with $\sin^2 \theta = \delta$.
	As shown above, $\lambda_{\theta} = 0$, and $K = (1-\delta/3)R_g + 2\delta w_g$.
	Let $\epsilon_j = 1/j$.
	By assumption, for each $j$, we have a solution $(A_j, \Phi_j)$ to (\ref{psw lebrun}) with $\epsilon = \epsilon_j$.
	We abbreviate $\sigma(\Phi_j)$ as $\sigma_j$ and $\beta(|\sigma(\Phi_j)|)$ as $\beta_j$.
	
	We first show that $\Phi_j$ does not strongly $L^p$-converge to $0$ as $j \to \infty$.
	If it does, then, after passing to a subsequence (still denoted by $j$) if necessary, $(A_j, \Phi_j)$ converges weakly to a reducible solution to (\ref{psw lebrun}) with $\epsilon = 0$, which contradicts our assumption.
	Consequently, $\|\sigma_j\|_2$ is uniformly bounded from below.
	
	We next show that $(1 - \beta_j |\sigma_j|)K_- |\sigma_j|$ strongly $L^1$-converges strongly to $0$ as $j \to \infty$.
	Equality in Corollary~\ref{theorem: LeBrun} implies, as in the proof of Theorem \ref{thm: general lebrun}, that
	\begin{equation} \label{ineq: equality}
		\begin{aligned}
		\int_X K_- |\sdh| \,d\mu_g &\le \int_X \beta_j (K_- + \epsilon_j) \sigma_j \wedge \sdh \\
		&\le \int_X \beta_j (K_- + \epsilon_j) |\sigma_j||\sdh| \,d\mu_g \le \int_X (K_- + \epsilon_j) |\sdh| \,d\mu_g.
		\end{aligned}
	\end{equation}
	Rearranging these inequalities, we obtain
	\[
		0 \le \int_X (1 - \beta_j |\sigma_j|) K_- |\sdh| \,d\mu_g \le \epsilon_j \int_X |\sdh| \,d\mu_g.
	\]
	Therefore, $(1 - \beta_j|\sigma_j|) K_- |\sdh|$ strongly $L^1$-converges to $0$; after passing to a subsequence (still denoted by $j$) if necessary, $(1 - \beta_j|\sigma_j|) K_- |\sdh|$ converges to $0$ almost everywhere.
	Since the nodal set of the non-trivial harmonic form $\sdh$ is of Lebesgue measure zero~\cite{MR1473317}*{Corollary 1}, it follows that $(1 - \beta_j|\sigma_j|) K_-$ converges to $0$ almost everywhere.
	Although $|\sigma_j|$ might be unbounded as $j \to \infty$, we have $(1 - \beta_j|\sigma_j|) |\sigma_j| = 0$ at each point where $|\sigma_j| \ge 1$ by our choice of $\beta$; hence, $0 \le (1 - \beta_j|\sigma_j|)|\sigma_j| K_- \le K_-$.
	In summary, $(1 - \beta_j|\sigma_j|)K_- |\sigma_j|$ is uniformly bounded and converges to $0$ almost everywhere; consequently, it strongly $L^1$-converges to $0$ by the Lebesgue dominated convergence theorem.
	
	The inequality (\ref{ineq: key}) implies that
	\[
		2\int_X |\nabla |\sigma_j||^2 \,d\mu_g \le \int_X (1-\beta_j|\sigma_j|) K_- |\sigma_j| \,d\mu_g.
	\]
	Thus, $\nabla |\sigma_j|$ strongly $L^2$-converges to $0$.
	Now the inequality (\ref{ineq: key}) again implies, after passing to a subsequence (still denoted by $j$) if necessary, $\hat{\sigma}_j := \sigma_j / |\sigma_j|$ strongly $L^2_1$-converges to a non-trivial smooth $g$-self-dual harmonic $2$-form $\hat{\sigma}_{\infty}$ with pointwise unit length; it is $g$-parallel if $\delta < 1$.
	The inequality (\ref{ineq: equality}) shows that $(\hat{\sigma}_{\infty}, \sdh) = |\hat{\sigma}_{\infty}| |\sdh|$ at any point in $X$;therefore, $\hat{\sigma}_{\infty}$ is a positive constant multiple of $\omega$.
	The theorem follows.
\end{proof}
	
\begin{remark}
	Let us mention how various curvature inequalities are derived from Corollary~\ref{theorem: LeBrun}.
	Setting $\delta = 0$ yields ~\cite{MR1487727}*{Theorem 3}, and setting $\delta = 1$ does~\cite{MR1919897}*{Theorem 3.3}.
	Considering $[\sdh] = - c_1^+(\s)$ and using the Cauchy-Schwartz inequality, we obtain
	\[
		 32\pi^2 (c_1^+(\s))^2 \le \int_X ((1-\delta/3)R_g + 2\delta w_g)^2 \,d\mu_g,
	\]
	which is equivalent to \cite{MR1822361}*{Theorem 4}.
	Then, setting $\delta = 0$ gives~\cite{MR1359969}*{Theorem 2}, and setting $\delta = 1$ does~\cite{MR1872548}*{Theorem 2.4}.
\end{remark}

\begin{acknowledgement}
	The authors wish to express their gratitude to H. Sasahira and M. Ishida for many stimulating conversations on various aspects of this work.
	They also gratefully acknowledge the many helpful suggestions of the anonymous referee.
\end{acknowledgement}


\end{document}